\documentclass[12pt,twopage]{paper}

\usepackage{libertine}
\usepackage{amssymb}
\usepackage{amsfonts}
\usepackage{amsmath}
\usepackage{bold-extra}
\usepackage{pifont}
\usepackage{yfonts}
\usepackage{mathrsfs}
\usepackage{url}
\usepackage{hyperref}
\usepackage{fancyhdr}

\usepackage{pmboxdraw}
 \usepackage{stmaryrd}

\newtheorem{theorem}{\textbf{\textsc{Theorem}}}[section]
\newtheorem{definition}[theorem]{\textbf{\textsc{Definition}}}
\newtheorem{lemma}[theorem]{\textbf{\textsc{Lemma}}}
\newtheorem{corollary}[theorem]{\textbf{\textsc{Corollary}}}
\newtheorem{proposition}[theorem]{\textbf{\textsc{Proposition}}}

\newenvironment{proof}
{\noindent\mbox{\textsf{\textbf{\textsc{Proof}}}:}}
{\hfill{\scriptsize \mbox{\underline{\texttt{\em QED}\,}$\!\big|$}}\bigskip}

\title{On a fallacy concerning I-am-unprovable  sentences:\\ what to take home from G\"odel's introduction}

\author{
{\sc Kaave \ Lajevardi}\\[-1ex]
\textrm{\large La Soci\'et\'e des Philosophes Ch\^{o}meurs, P.O.Box~13197-73587, T\'eh\'eran, IRAN.}\\[-1ex]
\textrm{\large E-mail:~\textsf{ kaave.lajevardi@gmail.com}}\\
{\sc Saeed \  Salehi  } \\[-1ex]
\textrm{\large Research Center of Biosciences \& Biotechnology (RCBB),   University of Tabriz,}\\[-1ex] \textrm{\large   P.O.Box~51666--16471, Tabriz, IRAN.  \, E-mail:~\textsf{root@SaeedSalehi.ir}} }

\begin{document}

\maketitle

\pagestyle{fancy}
\lhead[ ]{\thepage.\quad{\sc K. Lajevardi \& S. Salehi}}
\chead[ ]{ }
\rhead[ ]{ {\em On a Fallacy Concerning I-Am-Unprovable  Sentences} }
\lfoot[ ]{ }
\cfoot[ ]{ }
\rfoot[ ]{ }

\begin{abstract}

\vspace{-1em}
We demonstrate that, in itself and in the absence of extra premises, the following argument scheme is fallacious: {\em The sentence A says about itself that it has a certain property F, and A does in fact have the property F; therefore A is true.} We then examine an argument of this form in the informal introduction of G\"odel's classic (1931) and examine some auxiliary premises which might have been at work in that context. Philosophically significant as it may be, that particular informal argument plays no r\^{o}le in G\"odel's technical results.
Going deeper into the issue and investigating truth conditions of G\"odelian  sentences (i.e., those sentences which are provably equivalent to their own unprovability) will provide us with insights regarding the philosophical debate on the truth of G\"odelian  sentences of systems---a debate
which is at least as old as Dummett (1963).

\bigskip


%
\end{abstract}
%


\section{Introduction.}\label{sec:intro}
Consider the following argument scheme, wherein boldface italic letters are placeholders for sentences and predicates of a sufficiently rich language:
\begin{itemize}\itemindent=1em
\item[$(1)$] $\textit{\textbf{A}}$ says about itself that it is an $\textit{\textbf{F}}$
\item[$(2)$] $\textit{\textbf{A}}$ is indeed an $\textit{\textbf{F}}$
\item[]   {\em therefore}
\item[$(3)$] $\textit{\textbf{A}}$ is true.
\end{itemize}
Harmless and perhaps trifling as the scheme $(1)$-$(2)$-$(3)$ may look, we will argue that it is in fact invalid: there are instances of $\textit{\textbf{A}}$ and $\textit{\textbf{F}}$ which make premises $(1)$ and $(2)$ true while the conclusion $(3)$ false. We will demonstrate the invalidity of the scheme and draw some morals hopefully attractive to the philosophical community.\footnote{So far as we know, the fallacy was first exposed in \cite{LajevardiSalehi2019}. That work does not offer any diagnosis of the fallacy and minimizes the exegesis of \cite{Godel1931}. While we sharpen some of the results of \cite{LajevardiSalehi2019} and apply them to the philosophical debate on the truth of G\"odelian sentences, we do not presume any knowledge of that paper.}

What is likely to make this fallacy of more interest is that it seems to have been committed by a logician no less than G\"odel, and then somewhere in the introduction to his ground-breaking 1931 paper---but let us add immediately that we are perfectly aware that the fallacy is of no consequence to G\"odel's technical results. That this {\em is} a fallacy is demonstrated in Section \ref{sec:details} below; whether or not {\em G\"odel} has committed it is somehow a matter of judgement and depends on what one is willing to read into G\"odel's text (see \ref{subsec:omega} below). At any rate, we think awareness of the fallacy provides us with new insights concerning arguments about the truth of G\"odelian sentences of axiomatic systems.

In what follows, we first (\S\ref{sec:inval}) give a rough exposition of why the $(1)$-$(2)$-$(3)$ scheme fails to be valid, followed by technical details in \S\ref{sec:details}. Section \ref{sec:classic} is an exegetical discussion of G\"odel's 1931 paper. Among other things, we consider a number of points that could be made in defence of G\"odel's use of a particular instance of the $(1)$-$(2)$-$(3)$ scheme. In \S\ref{sec:gsent}, we exploit our results to say something which we believe systematize a number of previous attempts by philosophers and logicians concerning the truth of self-referential sentences which assert their own unprovability. The Appendices 
present more abstract versions of some of the issues presented in the paper.

Let us make it explicit that here we are {\em not} dealing with self-referential sentences as such (see \ref{subsec:irrev} below). Rather, our subject matter is, first and foremost, those sentences which are provably equivalent in a suitable ambient theory to their own unprovability relative to that same theory---hence the `I-am-unprovable' of our title.

\section{Invalidation: an overview}\label{sec:inval}
We should, of course, answer two questions prior to embarking on the invalidation. Our answers in the present section are not meant to be rigorous or complete---we only wish to give a sense of what is going on.
\begin{itemize}\itemindent=1em
\item[$(\alpha)$] What is it for a sentence to say something about itself?
\item[$(\beta)$] What is it for a sentence to be true? In particular: What is it for a sentence to have
       a property  {\em indeed}?
\end{itemize}
In the context of modern mathematical logic, here are the answers: $(\alpha)$ a sentence says something about itself just in case a certain biconditional, containing the sentence and a predicate applied to a name of that sentence, is a theorem of the background theory of the context.\footnote{See subsections \ref{subsec:g} and \ref{subsec:irrev} below.} As for $(\beta)$, a sentence is true just in case it holds (i.e., satisfies Tarski's standard recursive definition) in a fixed designated structure. These will become clearer in the paragraph below. And how could $(1)$ and $(2)$ fail to imply $(3)$? The how-is-it-possible question is relatively easy to answer. Here is an overview:

What a sentence $\psi$ {\em says} needs an interpretative or translational work to tell, and all interpretation is theory-relative.\footnote{We use the term `theory' in the standard logical sense: a {\em theory} is a set (usually a recursively enumerable one) of formul{\ae}  none of which contain any free variables.} Any such work takes place in a theory $T$ in a given language, a language which either contains $\psi$ itself or has a means of referring to it or mentioning it, e.g., via having a name or a code for $\psi$. Now if, according to $T$, the sentence $\psi$ is equivalent to $\Delta(\#\psi)$, where $\#\psi$ is a name for $\psi$ in the language of $T$ and $\Delta$ is a predicate of the same language, we say that $\psi$ says about itself that it is a $\Delta$. Thus, looking at G\"odel's construction in his classic 1931 paper, we have a sentence $G$ such that ${\rm PM}$ (the system of {\em Principia Mathematica} of Russell and Whitehead) proves the biconditional $G\!\leftrightarrow\!\neg{\rm Pr}(\#G)$, where Pr is a standard provability predicate for ${\rm PM}$. It is in virtue of this that one thinks of $G$ as {\em saying about itself that it is unprovable}. If the context makes it clear what the background theory is, or if the discussion is carried out rather informally and the theory is more or less the set of all we take for granted in that discussion, the theory is left unspecified.

Suppose that our background theory $T$ is not ``sound''. That is to say, suppose that some of the theorems of $T$ are not true, where {\em true} is relative to a given structure as said in $(\beta)$ above.\footnote{Normally, when one is talking about the sentences of the language of arithmetic, `true' simpliciter means true in the standard model $\mathbb{N}$ of natural numbers---see e.g., \cite[p.~145$n$4]{Godel1931} 
quoted below in \ref{subsec:indeed}.} It may so happen that amongst the false theorems of $T$, one is of the form $\psi\!\leftrightarrow\!\Delta(\#\psi)$. If so, then, according to $T$, the sentence $\psi$ says about itself that it is a $\Delta$. Our unsound theory is just in error about what $\psi$ ``really says''.

Now if we somehow manage to make all these happen for a false $\psi$ which is {\em in fact} a $\Delta$ (i.e., in such a way that $\Delta(\#\psi)$ is true in $\mathbb{N}$), we are done with showing the invalidity of $(1)$-$(2)$-$(3)$: $\psi$ says about itself that it is a $\Delta$ (in the eye of $T$), and $\psi$ is indeed a $\Delta$; however, $\psi$ is false. Hence the invalidity.

\section{Invalidation: the details.}\label{sec:details}
We show how to find a triplet of a theory $T$, a sentence $\psi$, and a predicate $\Delta$ as a counterexample to the $(1)$-$(2)$-$(3)$ scheme.

\begin{itemize}\itemindent=2.5em
  \item[{\bf Step 1.}] Let $T$ be a consistent but unsound theory in the language of arithmetic.\footnote{A well known example of such theory is Peano's Arithmetic (${\rm PA}$) to which a statement of the inconsistency of ${\rm PA}$ is added---that is to say,  ${\rm PA}+\neg{\rm Con(PA)}$ is consistent (by G\"odel's second incompleteness theorem) but not sound {\em if} ${\rm PA}$ is consistent to begin with.
 \par \quad
For an inconsistent $T$, it is all too easy to invalidate the $(1)$-$(2)$-$(3)$ argument scheme. In \ref{subsec:omega} we will invalidate $(1)$-$(2)$-$(3)$ even for $\omega$-consistent theories.} For reasons which need not concern us here, it is a normal practice today to assume that $T$ is an extension of Robinson's Arithmetic  $\textit{Q}$.\footnote{$\textit{Q}$ is basically a first-order description of the algebraic properties of $+, \times, 0$, and $1$ in $\mathbb{N}$, with no axiom scheme of induction. This theory is known to be more than enough for the purpose of proving the first incompleteness theorem.} As $T$ is unsound, there is a false sentence $\varphi$ which is provable in $T$; i.e., there is a $\varphi$ with $T\vdash\varphi$ and $\mathbb{N}\nvDash\varphi$. Let $\Delta(x)$ be any formula whatsoever with exactly one free variable $x$. Apply the celebrated diagonal lemma \cite[p.~173]{BoolosJeffrey1989} to the formula $\varphi\!\leftrightarrow\!\Delta(x)$, to get a sentence $\psi$ such that\footnote{Cf. \cite[p.~238]{McGee1992} and \cite[p.~128]{Cook2006}.} $$\textit{Q} \vdash  \psi\!\leftrightarrow\![\varphi\!\leftrightarrow\!
\Delta(\#\psi)],$$ where $\#\psi$ is the canonical term for the G\"odel number of $\psi$. By propositional logic and our assumption that $T \vdash\varphi$, we have
$$(\ast) \quad T \vdash \psi\!\leftrightarrow\!\Delta(\#\psi).$$
On the other hand, by the well known textbook proof of the diagonal lemma as presented in \cite{BoolosJeffrey1989}, we have
$$\mathbb{N}\vDash\psi\!\leftrightarrow\![\varphi
\!\leftrightarrow\!\Delta(\#\psi)],$$
so that, because of the falsity of $\varphi$ in $\mathbb{N}$, we get
$$(\ast\!\ast) \quad  \mathbb{N}\vDash\psi\!\leftrightarrow\!\neg
\Delta(\#\psi).$$
  \item[{\bf Step 2.}] For every $\Delta$, our ($\ast$) enables us to make the first premise $(1)$ true. We need to take a bit of care to make sure that $(2)$ becomes true and $(3)$ false, and $(\ast\!\ast)$ helps us in taking care of them simultaneously. An obvious choice would be taking $\Delta(x)$ to be any tautological predicate like `$x=x$'\footnote{Cf. \cite{Leitgeb2002}.}, as $\Delta(\#\psi)$ will then hold up in $\mathbb{N}$ so that the invalidation will be completed because of $(\ast\!\ast)$. However, as tautologies are perhaps not immensely interesting, let us offer two other options, namely let $\Delta_1(x)$ be ${\rm Pr}_T(x)$, and let $\Delta_2(x)$ be $\neg{\rm Pr}_T(x)$, where ${\rm Pr}_T$ is a provability predicate for $T$.
  It is a pair of nice little exercises to show that the resulting sentences $\psi_1$ and $\psi_2$ are such that both $\Delta_1(\#\psi_1)$ and $\Delta_2(\#\psi_2)$ hold in $\mathbb{N}$.\footnote{{\em Solution to the exercises}.
\par  \qquad
For $\Delta_1$  and $\psi_1$: If   $T \vdash  \psi_1\!\leftrightarrow\! {\rm Pr}_T(\#\psi_1)$, then, by L\"ob's theorem  \cite[p.~187]{BoolosJeffrey1989}, we have
$T \vdash  \psi_1$, so that, because of arithmetization, ${\rm Pr}_T(\#\psi_1)$ holds in $\mathbb{N}$; hence  we have ${\rm Pr}_T(\#\psi_1)$ true and $\psi_1$ false, because of $(\ast\!\ast)$. Let us note that $T$ needs to be substantially richer than $\textit{Q}$ to have L\"ob's rule; it suffices for $T$ to contain ${\rm PA}$ or its fragment ${\rm I}\Sigma_1$.
\par \qquad
For $\Delta_2$  and $\psi_2$: if $T \vdash  \psi_2\!\leftrightarrow\! \neg{\rm Pr}_T(\#\psi_2)$ and $T$ is consistent, then, as demonstrated in G\"odel's proof of his first incompleteness theorem, $\psi_2$ is $T$-unprovable, hence,  by $(\ast\!\ast)$, $\neg{\rm Pr}_T(\#\psi_2)$ is true while $\psi_2$ false.}
\end{itemize}
The invalidation is now accomplished. We have introduced a procedure for producing instances of $\textbf{\textit{A}}$ and $\textbf{\textit{F}}$ showing the invalidity of the $(1)$-$(2)$-$(3)$ scheme. Specifically, for each unsound theory $T$ we have introduced two particular cases in point: first, a {\em false} sentence which says about itself that it is provable, and it is provable indeed; second, another {\em false} sentence which says about itself that it is unprovable and is unprovable indeed.

\subsection{Digression: another invalid argument scheme.}\label{subsec:dig}
Insofar as one is inclined to think that $(3)$ follows from $(1)\&(2)$ on the basis that a sentence is true when it has the property it ascribes to itself, one may also be inclined to think that $(2)$ follows from $(1)\&(3)$ on the basis that a true self-referential sentences does have the property it ascribes to itself---the $(1)$-$(2)$-$(3)$ scheme seems to be {\em as valid as} the following:
\begin{itemize}\itemindent=1em
\item[$(1')$] $\textit{\textbf{A}}$ says about itself that it is an $\textit{\textbf{F}}$
\item[$(2')$] $\textit{\textbf{A}}$ is true
\item[]   {\em therefore}
\item[$(3')$] $\textit{\textbf{A}}$ is indeed an $\textit{\textbf{F}}$.
\end{itemize}
Now the cute thing about the two-step invalidation presented in this section is that its first step suffices to show, in one breath, that {\em at least one of the two schemes is invalid}.\footnote{Note that by virtue of $(\ast)$, $\psi$ says about itself that it is a $\Delta$ while $(\ast\!\ast)$ shows that $\psi$ is true just in case $\psi$ is {\em not} a $\Delta$. Now if $\psi$ is indeed a $\Delta$, then $\psi$ is not true and $(1)$-$(2)$-$(3)$ is invalid; if, on the other hand, $\psi$ is {\em not} indeed a $\Delta$, then $\psi$ is true and $(1')$-$(2')$-$(3')$ is invalid. Therefore, at least one of $(1)$-$(2)$-$(3)$ or $(1')$-$(2')$-$(3')$ is invalid. Above we saw that the first scheme is invalid; in a moment we will see that so is the second one.}
  Assuming, intuitively, that deriving $(3)$ from $(2)$ looks as good as doing the reverse, this would show that there must be something wrong with the original $(1)$-$(2)$-$(3)$ argument scheme even if we do not go through Step 2 above.

      Let us invalidate the $(1')$-$(2')$-$(3')$ scheme right away.
Suppose $T$ is a consistent and unsound theory that proves a false sentence $\varphi$. Take $\psi$ to be the true sentence $\neg\varphi$, and take $\Delta(x)$ to be $\neg{\rm Pr}_T(\#\neg x)$. Now $\Delta(x)$ is true iff the sentence with the G\"odel number $x$ is consistent with $T$. Since $T\vdash\varphi$, we have $T\vdash\neg\psi$, and so $T\vdash{\rm Pr}_T(\#\neg\psi)$. Therefore $T \vdash\neg\psi
\!\leftrightarrow\!{\tt Pr}_T(\#\neg\psi)$, which implies that $T \vdash\psi
\!\leftrightarrow\!\Delta(\#\psi)$. Thus $\psi$ says about itself that it is a $\Delta$ (in the eye of $T$), and $\psi$ is true; but $\Delta(\#\psi)$ is not true since by the $T$-provability of $\neg\psi$, ${\rm Pr}_T(\#\neg\psi)$ is true.

\section{G\"odel's 1931 paper.}\label{sec:classic}
In this exegetical section, we examine what we think of as an instance of the $(1)$-$(2)$-$(3)$ argument scheme occurring in the introductory and informal section of  \cite{Godel1931}, and make a number of comments---in particular, we explore a number of extenuating circumstances in favour of G\"odel's informal argument. Once again, we wish to make it explicit that the invalidity of the $(1)$-$(2)$-$(3)$ scheme does {\em not} affect the correctness of G\"odel's mathematical results in the technical part of the 1931 paper---though it may, we hope, shed some light on some philosophical issues, to be discussed below in Section~\ref{sec:gsent}.

G\"odel begins his 1931 paper with an introductory section, wherein, in a very lucid and reader-friendly way, he informally introduces the ideas behind the first incompleteness theorem. By the end of the antepenultimate paragraph of that section, he has argued for the
independence of what is now called `the G\"odel sentence' of ${\rm PM}$. His informal proof is quite short, for two reasons: first, he does not go into the details of arithmetization and  simply assumes that arithmetization can be done. Secondly, he assumes that ${\rm PM}$ is sound. The use of the second assumption makes the result actually weaker than what is stated in the technical part of the paper, where he proves a stronger version by means of substituting the soundness assumption by $\omega$-consistency (see \ref{subsec:omega} below).

In his next paragraph, G\"odel comments on the analogy of the argument with some antimonies. He could have stopped there---we regret that he did not, for we think the next paragraph contains a fallacy. In the final paragraph of the section, he writes \cite[p.~151, italics in the original]{Godel1931}:
\begin{itemize}
\vspace{-1ex}
\item[]\begin{quote}
\begin{quote}
{\small
From the remark that $[R(q);q]$ says about itself that it is not provable, it follows  at     once  that $[R(q);q]$ is true, for $[R(q);q]$ {\em is} indeed unprovable (being undecidable).}	
\end{quote}
\end{quote}
\vspace{-1ex}
\end{itemize}
The way that the proposition is named by G\"odel (i.e., `$[R(q);q]$') reflects its manner of construction, which we will discuss below; however, for the ease of exposition, let us call the sentence `$G$'. The argument then seems to be this: {\em $G$ says about itself that it is not provable, and $G$ is indeed unprovable; therefore $G$ is true}, which is an instance of the $(1)$-$(2)$-$(3)$ scheme introduced at the beginning of our paper. We have argued in Section~\ref{sec:details} above that, {\em as it stands}, the argument is invalid. In what follows will try to be maximally charitable to G\"odel. Besides, we are aware that G\"odel himself has said, both in the 1931 paper and elsewhere, that in his introduction he does not claim to be perfectly precise.\footnote{ \cite[p.~147]{Godel1931};  cf. \cite[pp.~88$f$]{Dawson1984} on the Zermelo-G\"odel correspondence. For an old critical comment on G\"odel's introduction see \cite[p.~58]{Helmer1937}, where Helmer writes ``G\"odel made one `mistake', namely that of writing an introduction to his paper ...''.  It seems to us that the point we are making  here has not been noted by these logicians.}

\subsection{The construction of $G$, whereby G\"odel guarantees the truth of the relevant instance of premise (1).}\label{subsec:g}
What we nicknamed `$G$'  is formally constructed in \cite[pp.~173$f$]{Godel1931}, where its official description there is `$17\, {\rm Gen}\, r$'. From the construction of $G$ it is quite clear that if $T$ is the background theory (be it ${\rm PM}$, ${\rm ZF}$, or what have you) with provability predicate ${\rm Pr}_T$, then
\begin{itemize}\itemindent=1em
\item[$(4)$] $T\vdash G\!\leftrightarrow\!\neg{\rm Pr}_T(\#G)$
\end{itemize}
which, for a modern reader, is a straightforward application of the diagonal lemma to the formula $\neg{\rm Pr}_T(x)$.\footnote{Though G\"odel is evidently exploiting a diagonal technique in the construction of $G$, it would be somehow anachronistic to say that he is using the {\em diagonal lemma}: it was Rudolf Carnap, in his 1934 {\em Logische Syntax der Sparche}, who first isolated the lemma as such---see \cite[p.~363$n$23]{Godel1934}.}  A quick look at a standard proof of the lemma as presented in \cite[p.~173]{BoolosJeffrey1989} shows that we also have
\begin{itemize}\itemindent=1em
\item[$(5)$] $\mathbb{N}\vDash\neg{\rm Pr}_T(\#G)$;
\end{itemize}
that is to say, $G$ is unprovable indeed. And what does G\"odel mean by {\em indeed}?

\subsection{G\"odel's `indeed'.}\label{subsec:indeed}
When G\"{o}del says that something is true, he means that it is true in the standard model $\mathbb{N}$. This is multiply evidenced by his explicit note that the formul{\ae} of the language in question should be understood in a way that in those formul{\ae},
\begin{itemize}
\vspace{-1ex}
\item[]\begin{quote}
\begin{quote}
{\small
no other notions occur but $+$ (addition) and $\boldsymbol\cdot$ (multiplication),  both for {\em natural}   {\em numbers},  and in which the quantifiers $(x)$, too, apply to {\em natural numbers} only.}  \;  \cite[p.~145$n$4, emphasis ours]{Godel1931}.	
\end{quote}
\end{quote}
\vspace{-1ex}
\end{itemize}
Thus by saying that $G$ is {\em indeed} unprovable, G\"odel surely means (5)---namely, that $\neg{\rm Pr}_T(\#G)$ holds in $\mathbb{N}$. Once again, the construction of $G$ guarantees that (5) is the case.

All these, we submit, show that the informal argument of the passage that quoted at the beginning of this section is
\begin{itemize}\itemindent=1em
\item[(4)] $T\vdash G\!\leftrightarrow\!\neg{\rm Pr}_T(\#G)$
\item[(5)] $\mathbb{N}\vDash\neg{\rm Pr}_T(\#G)$
\item[]   {\em therefore}
\item[(6)] $\mathbb{N}\vDash G$,
\end{itemize}
whose logical form we recognize as $(1)$-$(2)$-$(3)$. G\"odel says that the conclusion follows `at once' (`sofort' in the original German) which, to our ears, suggests that no extra assumptions are needed. But perhaps he did have additional assumptions in mind? Let us see.

\subsection{What G\"odel officially assumed: not soundness, but $\omega$-consistency.}\label{subsec:omega}
G\"odel's informal argument is fallacious {\em only if} the background theory $T$ is not sound. For if all the theorems of $T$ are true in $\mathbb{N}$, so is the biconditional of (4), which, together with (5), makes $G$ true in $\mathbb{N}$, as asserted by (6). The assumption of soundness may seem to be what G\"odel had in mind, for he says, in the very first paragraph of his paper, that his result holds in particular for every extension of ${\rm PM}$,
\begin{itemize}
\vspace{-1ex}
\item[]\begin{quote}
\begin{quote}
{\small
provided no false proposition of the kind  specified
…   become   provable   owing to the added axioms.}	\;
\cite[pp.~145$f$]{Godel1931}
\end{quote}
\end{quote}
\vspace{-1ex}
\end{itemize}
In itself, this suggests that G\"odel is talking about sound theories. Also, in the next-to-last paragraph of his Section 1, he says \cite[p.~151]{Godel1931} that his method of proof is applicable to any formal system which is, first, of sufficient expressive power and, secondly, is such that ``every provable formula is true in the interpretation considered''. This, too, suggests that he is talking about sound theories only.

However, soundness is definitely {\em not} what G\"odel had in mind as a required property of the theories under consideration in his paper, for he immediately adds,
\begin{itemize}
\vspace{-1ex}
\item[]\begin{quote}
\begin{quote}
{\small
The purpose of carrying out the above proof with full precision in what follows is,   among other things, to replace {\em the second of the assumptions just mentioned} by a   purely formal and much weaker one.}	 \;   (Ibid., emphasis added.)
\end{quote}
\end{quote}
\vspace{-1ex}
\end{itemize}
That much weaker and purely formal assumption is $\omega$-consistency, so baptized by G\"odel himself, which is introduced just before the statement of the first incompleteness theorem and is mentioned in it \cite[p.~173]{Godel1931}.\footnote{By definition, a theory $T$ (in the language of arithmetic) is {\em $\omega$-consistent} iff there is no formula $\xi(x)$ such that $T$ proves the sentence $\exists x\neg\xi(x)$ {\em and} proves all of the following sentences: $\xi(\textbf{1}), \xi(\textbf{2}), …, \xi(\textbf{n}), \ldots$ (for every natural number $n$), where boldface roman characters denote the standard terms for the corresponding numerals. (For the formal definition see \cite[p.~173]{Godel1931} or \cite[pp.~851$f$]{Smorynski1977}.)
\par\quad
Simple consistency of ${\rm PM}$ suffices for showing that $G$ is ${\rm PM}$-unprovable. It seems that the sole purpose of introducing $\omega$-consistency was that G\"odel was unable to show the irrefutability of $G$ merely by assuming consistency. The assumption of $\omega$-consistency was weakened to simple consistency in  \cite{Rosser1936}, where Rosser showed the undecidability of {\em another} sentence.
}
So, although G\"odel talks about truth and soundness in his informal and introductory Section 1, he avoids talking about it in the technical part of the paper---in fact, he explicitly tells us that one purpose of going formal and precise was just replacing his talk about {\em truth} of the theorems of certain theories (say ${\rm PM}$) with the talk about a purely syntactical property of them.

Officially speaking, then, G\"odel's first incompleteness theorem assumes that the theories in question are $\omega$-consistent, a condition weaker than soundness.
An invaluable source here is \cite{Isaacson2011}, wherein we find a theorem---Isaacson's Proposition 19, attributed to Kreisel---telling us about a false sentence $K$ such that ${\rm PA}+K$ is $\omega$-consistent.\footnote{This is an almost immediate consequence of formalizing the notion of $\omega$-consistency. See also \cite[p.~36]{Lindstrom1997}.}
This allows us to invalidate the $(1)$-$(2)$-$(3)$ scheme via presenting an $\omega$-consistent theory: Simply take $\Delta(x)$ to be the formula `$x=\#K$', and take $\psi$ and $T$ to be the sentence $K$ and the theory ${\rm PA}+K$, respectively. Moral: the assumption of $\omega$-consistency of the background theory is not strong enough to save the $(1)$-$(2)$-$(3)$ scheme from invalidity.

Having said that, we acknowledge that we {\em might}  have invalidated G\"odel's argument only via over-generalizing it.
Several additional premises (including some restrictions on the complexity of the relevant predicates and formulæ) are examined in  \cite{LajevardiSalehi2019}
which render the scheme valid.
 More specifically, if $\textit{\textbf{A}}$ and $\textit{\textbf{F}}$ are of $\Pi_1$ complexity, then the extra assumption of $\omega$-consistency guarantees the validity of the scheme---see Proposition~\ref{prop} in the Appendix (\S\ref{app:c}). Both $G$ and $\neg{\rm Pr}_T(x)$  are $\Pi_1$ for all the theories we have mentioned in this paper.

Did G\"odel actually commit the fallacy? Perhaps not: perhaps he meant to show the truth of $G$ only for sound theories and did not intend to assert the truth of $G$ in general for all the theories which are subject to the official version of his first incompleteness theorem. Alternatively, he might have been careless to mention that $G$ and $\neg{\rm Pr}_T(x)$  are $\Pi_1$ (see footnote~\ref{fn}  below for G\"odel's terminology).
Though such considerations clear G\"odel himself of the charge of talking fallaciously, they will not acquit a number of philosophers who talked about ``seeing'' the truth of $G$. We return to this in \ref{subsec:truthg} below.
\subsection{The (ir)relevance of the notion of self-reference.}\label{subsec:irrev}
There is an ongoing discussion on what it means for a sentence  to say something about itself---see   \cite{HalbachVisser2014}. By no means do we want to neglect this scholarly issue; however, if the task is to evaluate {\em G\"odel's} argument as presented in his introductory section, we take it to be quite obvious that he had thought of the very sentence $G$ as a sentence saying about itself that it is unprovable---this sentence  is, in fact, a paradigm of self-referentiality. So, whatever the correct analysis of the concept of self-attribution might turn out to be, {\em what G\"odel, qua the author of} \cite{Godel1931}, {\em had in mind} must be something retrievable from what he has done in that paper. Given that his apparatus for constructing $G$ is a version of the diagonal lemma, we think there is no choice but to think that, for G\"odel, $\textit{\textbf{A}}$'s saying about itself that it has property $\textit{\textbf{F}}$ just means that the biconditional $\textit{\textbf{A}}\!\leftrightarrow\!
\textit{\textbf{F}}(\#\textit{\textbf{A}})$ is provable in the system. We therefore find the $(4)$-$(5)$-$(6)$ argument the only thing that G\"odel could have had in mind in this connection, the logical form whereof we recognize as the $(1)$-$(2)$-$(3)$ scheme.

We do agree with the statement that
is quoted in \cite[p.~672]{HalbachVisser2014} sympathetically
from logician Craig Smory\'nski, that the  ``notion of a sentence's expressing something about itself has not proven fruitful'', and we think Kripke is right when, concerning the argument for the independence of $G$, says
\begin{itemize}
\vspace{-1ex}
\item[]\begin{quote}
\begin{quote}
{\small
The argument can be carried through without noticing explicitly  that    $G$ says something about itself (i.e. that it itself is not provable).} \; \cite[p.~239]{Kripke2014}.
\end{quote}
\end{quote}
\vspace{-1ex}
\end{itemize}
Here we wish to make two minor comments. First, these insightful remarks are made in the context of {\em formal} mathematical logic: Smory\'nski and Kripke justly remark that, insofar as G\"odel's technical results are concerned, G\"odel did not have to talk about self-reference. It should be clear, however, that an evaluation of the {\em informal} argument presented in the introductory part of G\"odel's paper requires a working notion of self-reference. Once again, we cannot see what G\"odel could have had in mind other than an informal variant of what we have presented as $(4)$-$(5)$-$(6)$.

Secondly, if, for whatever reason, G\"{o}del felt an urge to talk about self-reference (perhaps in order to make his introduction more attractive to the general reader), with the wisdom of hindsight we know that instead of self-reference in the standard textbook manner, which for every formula $\Delta(x)$ provides us with a sentence $\psi$ such that the biconditional $\psi\!\leftrightarrow\!\Delta(\#\psi)$ is a theorem of the background theory (and let us call this kind of self-reference {\em indirect}), he could have exploited a {\em direct} self-reference in such a way that, for every formula $\Delta(x)$ we get a sentence $\psi$ such that $\psi$ just {\em is} the formula $\Delta(\#\psi)$.\footnote{So far as we know, the idea of such a direct self-reference goes back to \cite[p.~693]{Kripke1975}. See also \cite[pp.~159$f$]{Visser2004}.}  Now, instead of extra assumptions
concerning the soundness or complexity of $\Delta$ and $\varphi$, here is another way of providing a valid version of the $(1)$-$(2)$-$(3)$ scheme: instead of a traditional indirect self-reference, one may go the direct way and replace $(1)$ with
\begin{itemize}\itemindent=2.5em
\item[$(1'')$] $\textit{\textbf{A}}$ is the sentence $\textit{\textbf{F}}(\#\textit{\textbf{A}})$,
\end{itemize}
which makes the scheme trivially valid.

\section{G\"odelian sentences and truth.}\label{sec:gsent}
\subsection{Why did G\"odel talk about truth?}\label{subsec:why}
The official statement of G\"odel's first incompleteness theorem \cite[p.~173]{Godel1931} does not mention truth: it is simply a theorem to the effect that every theory of a certain totally syntactically specified kind is {\em incomplete}, in the sense that there are sentences in the language of the theory which are axiomatically undecidable on the basis of that theory: they are neither provable nor refutable. If you are a kind of realist with respect to mathematics, as surely G\"odel was from the 1940s onwards (if not earlier), you are inclined to say that each undecidable sentence is ``really'' either true or false---hence the {\em popular} version of the first incompleteness theorem: for every theory of a certain specified kind, there are {\em true} sentences which are unprovable on the basis of that theory. Yet this is something which is not formally presented in G\"odel's Theorem.

Regardless of one's verdict on G\"odel in connection with the $(1$)-$(2)$-$(3)$ argument in his
 introductory section, one may observe that that section has two drawbacks in comparison with the technical parts of the paper:
  \begin{quote}
  \begin{itemize}\itemindent=1ex
 \item[I.] By posing a stronger requirement on the theories in question (soundness, instead of $\omega$-consistency), it weakens the {\em content} of the first incompleteness theorem, making it applicable to a narrower class of theories.
 \item[II.] By introducing the notion of truth (via soundness), it makes the {\em proof} of the theorem acceptable to a smaller class of mathematicians.
 \end{itemize}
 \end{quote}
We have already (in \ref{subsec:omega}) talked about $\omega$-consistency, hence substantiated (I). As for (II), note that G\"odel's formal and official proof, which is a {\em tour de force}  of giving a purely constructive and syntactical proof of the first incompleteness theorem, is acceptable to realists and anti-realists alike;\footnote{G\"odel himself comments \cite[p.~177$n$45a]{Godel1931} that his technical proof is intuitionistically acceptable.}  but, by talking about truth, G\"odel actually makes parts of the argument of his introduction susceptible to becoming unconvincing for those readers who are of intuitionist persuasion. In fairness to G\"odel, however, we admit that (I) and (II) are reasonable prices to pay---at least in an introductory section---to attain a higher level of perspicuity of exposition, attained by G\"odel's introduction.

However, a nagging question remains: why did G\"{o}del pause to argue that $G$ is true? Even if this particular instance of the $(1)$-$(2)$-$(3)$ scheme is saved from invalidity by the presence of the assumption of soundness in the context of G\"{o}del's introduction, it is not quite clear how he would restore the argument for the truth of $G$ in the context of unsound theories without appealing to the complexity issues (referred to in \ref{subsec:omega} above), which are absent from his 1931 paper.\footnote{\label{fn}We are far from saying that G\"{o}del was unaware of well-known and elementary facts about the complexity of formul{\ae}. He surely was---thus, for example, in a short note of the same period he talks about propositions ``of  the type of Goldbach'' \cite[p.~203]{Godel1931a}, by which here he means what is now called $\Pi_1$-formulas; see also \cite[p.~154]{Smith2013}.}  And it seems to us that, even for the sake of elementary lucid exposition  of topics he later deals
with in full precision, the question about the truth of $G$ is simply irrelevant. Why, then, did he care to talk about the {\em truth} of $G$?

By way of speculation, we submit that the reason may have something to do with what might be thought of as G\"{o}del's target or antagonist, the school of David Hilbert. Talking about the relationship between G\"{o}del's incompleteness theorems and Hilbert's programme goes well beyond the scope of this paper; we just want to make an amateurish conjecture that, after what he perhaps thought of as a fatal blow at Hilbert's programme via the mathematical content of his incompleteness theorems, G\"{o}del wanted to add a {\em coup de gr\^{a}ce} by arguing that while for every system of the specified kind the   sentence $G$ is axiomatically undecidable, it {\em is} decidable via some acceptable reasoning accessible to human beings;  hence mathematics cannot be captured by any axiomatics. We do regret that he has done so.

      Now let us talk more rigorously.

\subsection{On the truth of G\"odelian  sentences.}\label{subsec:truthg}
Concerning the age-old question of the truth of G\"odelian  sentences we think some insight can be gained from what we have investigated so far. Of the first-rate and/or well-known works on this question the following references readily come to mind: \cite{Dummett1963},
\cite{Smorynski1977}, \cite{Smorynski1985},
\cite{Boolos1990},
\cite{Milne2007},
\cite{Raatikainen2005},
\cite{Sereny2011},
\cite{Shapiro1998},
\cite{Tennant2002}, and the collection of papers 
\cite{HorstenWelch2016}. We wish to make two specific points before presenting our analysis of what is going on in a good number of such discussions.

\subsubsection{A consistent theory may have more than one G\"odelian  sentences---even up to truth-value.}\label{subsubsec:more}
Following \cite{LajevardiSalehi2021}, we choose to talk about {\em G\"odelian} sentences of a theory $T$ not {\em the G\"odel sentence} of $T$, for an unsound $T$ may have a true I-am-unprovable sentence as well as a false one (though they are equivalent in the eye of $T$), in which case it is bizarre to talk about {\em the} I-am-unprovable sentence. This, however, may be more a matter of propriety of speech than anything of great logico-mathematical significance.\footnote{The reader is invited to consult \cite{LajevardiSalehi2023} and the references therein for some discussions on the notion of {\em the} G\"odel sentence (of a theory, possibly  relative to a certain sub-theory).}
Here is the formal definition.
\begin{definition}
$G$ is a {\em G\"odelian sentence} of $T$ iff $T\vdash G\!\leftrightarrow\!\neg{\rm Pr}_T(\#G)$.
\end{definition}
\subsubsection{G\"odelian sentences of $T$ are collectively  true if and only if $T$ is sound.}\label{subsubsec:truesound}
For this fact we present a rigorous proof. In the following lemma and theorem we assume that $T$ satisfies L\"ob's derivability conditions.

\begin{lemma}\label{lem}
\textit{If $\tau$ is a $T$-provable sentence and $\gamma$ is a G\"odelian sentence of $T$, then $\tau\!\wedge\!\gamma$ is a G\"odelian sentence of $T$ as well.}
\end{lemma}
\begin{proof}
By the $T$-provability of $\tau$ we have
$T\vdash (\tau\!\wedge\!\gamma)
\!\leftrightarrow\!\gamma$.
Therefore, by derivability conditions, we have
$T\vdash{\rm Pr}_T(\#[\tau\!\wedge\!\gamma])
\!\leftrightarrow\!
{\rm Pr}_T(\#\gamma)$
hence
$T\vdash
\neg {\rm Pr}_T(\#\gamma)
\!\leftrightarrow\!
\neg {\rm Pr}_T(\#[\tau\!\wedge\!\gamma])$.
Since $\gamma$ is a G\"odelian sentence of $T$, we have
$$T\vdash
(\tau\!\wedge\!\gamma)\!\leftrightarrow\!
\gamma\!\leftrightarrow\!
\neg {\rm Pr}_T(\#\gamma)\!\leftrightarrow\!
\neg {\rm  Pr}_T(\#[\tau\!\wedge\!\gamma]),$$
which shows that $\tau\!\wedge\!\gamma$ is a G\"odelian sentence of $T$.
\end{proof}

Now we have

\begin{theorem}\label{thm}
\textit{$T$ is sound if and only if all G\"odelian sentences of $T$ are true.\footnote{See Proposition~\ref{thm-app} in the Appendix (\S\ref{app:a}) for a more general result.}
}\end{theorem}
\begin{proof}
If $\gamma$ is a G\"odelian sentence of a sound theory $T$, then  $T\vdash\gamma\!\leftrightarrow\!\neg {\rm Pr}_T(\#\gamma)$, which implies that the biconditional $\gamma\!\leftrightarrow\!\neg {\rm  Pr}_T(\#\gamma)$ is a true sentence. But by the $T$-unprovability of $\gamma$ (G\"odel's proof), the sentence $\neg {\rm Pr}_T(\#\gamma)$ is true; therefore $\gamma$ is true.

On the other hand, if $T$ is not sound then there is a false $T$-provable sentence $\tau$.  Let $\gamma$ be any G\"odelian sentence of $T$ (whose existence is demonstrated by the diagonal lemma). Then $\tau\wedge\gamma$ is a false sentence which is, by  Lemma~\ref{lem}, a G\"odelian sentence of $T$. Therefore $T$ has a false G\"odelian sentence.
\end{proof}

      Now, to the main business of this section.

\subsubsection{At the end of the day, should we say that G\"odelian sentences are true?}\label{subsubsec:arethey}
It seems to us that a good number of the works on the truth of G\"odelian sentences are shaky or even outright fallacious. This has been observed or claimed by some philosophers such as Shapiro~\cite{Shapiro1998}; yet we think we can offer a brief and more systematic analysis.
\begin{itemize}\itemindent=0.75em
  \item[\textbf{\textit{i}}.] Occasionally, one encounters an argument like this: {\em Each G\"odelian sentence is true because it says that it is unprovable and it is indeed unprovable (if the theory is consistent)}. In itself, this is a fallacy (as shown in our Section~\ref{sec:details} above). To say the least, the italicized argument needs more assumptions for its validity. It's a pity to see this fallacious-or-gappy argument in some early editions of first-rate textbooks such as \cite[p.~186]{BoolosJeffrey1989} and \cite[p.~159]{Mendelson1979}.\footnote{Happily, the mistakes are not there anymore. Concerning \cite{BoolosJeffrey1989}, it is noteworthy that from the fourth edition onwards (prepared by John P. Burgess), the fallacious passage is simply omitted. Mendelson, on the other hand, now shows the truth of G\"odelian sentences only for sound theories---thus on page 209 of the sixth edition, published in 2015, we have the assumption that the theory ``is a true theory''.}
  \item[\textbf{\textit{ii}}.] Let us not worry about the {\em reasonings} of those, like \cite{Dummett1963}, who argue that G\"odelian  sentences of consistent theories are true; let us ask if their {\em conclusion} is correct---that is to say: Does the consistency of a theory $T$ guarantee the truth of all its G\"odelian sentences? (We will assume that the theories in question are recursively enumerable extensions of Robinson's $\textit{Q}$.) Theorem 3.3 in  \cite{LajevardiSalehi2021} answers this affirmatively, with a proviso. Recall that $P$ is a G\"odelian sentence of $T$ iff the biconditional $P\!\leftrightarrow\!\neg {\rm Pr}_T(P)$ is provable in $T$. Now if, {\em moreover}, the biconditional is also true in the standard model $\mathbb{N}$, then $P$ is true if and only if $T$ is consistent. The condition on the truth of the biconditional  $P\!\leftrightarrow\!\neg {\rm  Pr}_T(\#P)$ is easily satisfied {\em if} $P$ is constructed via the celebrated diagonal lemma.
  \item[\textbf{\textit{iii}}.] \textsc{But}, not all G\"odelian  sentences of a theory need to satisfy this extra condition. The above argument does not show that every G\"odelian   sentence of a consistent theory is true; rather, it shows it only for those $P$s asserting their own unprovability such that the biconditional $P\!\leftrightarrow\!\neg {\rm Pr}_T(\#P)$ is also true in  $\mathbb{N}$.
\par
Enter our Theorem. Even if the reader does not share our worry concerning the impropriety of the term `the G\"odel sentence', he or she may concede this much: the argument mentioned in \ref{subsubsec:arethey}.ii   does not show that every sentence which is $T$-equivalent to its own $T$-unprovability is true if $T$ is consistent. If one's question is whether every sentence asserting its own unprovability is true, one should be noted that, because of  Theorem~\ref{thm} proved above, the truth of all such sentences is equivalent to the soundness of $T$, and---you may recall---soundness is {\em much stronger} an assumption than $\omega$-consistency, which is in turn stronger than simple consistency.\footnote{ It may even be demonstrated that every false sentence is a G\"odelian  sentence of some unsound theory (see the Appendix,\S\ref{app:b}).}
\end{itemize}

\section{Conclusion.}\label{sec:conc}
As a matter of logical fact, the $(1)$-$(2)$-$(3)$ argument scheme displayed at the beginning of this paper is invalid even if we assume that the background system is $\omega$-consistent (and {\em a fortiori}, even if the system is consistent). G\"odel's particular instance of it---the (4)-(5)-(6) argument---is  valid (even for unsound theories), but only because the involved sentence and predicate are of certain complexity---that is to say, because they are both $\Pi_1$, or ``of the type of Goldbach'' in G\"odel's later terminology. Interestingly, as we will see in the Appendix (\S\ref{app:c}), the way G\"odel constructs his self-referential sentence makes the premise (5) a consequence of his (4), and therefore redundant.

Apart from logic chopping, our observations are of some consequence concerning the debate on the truth of {\em G\"odelian sentences}---or, to avoid verbal disputes, the truth of {\em each and every sentence which asserts its own unprovability}. Our Theorem~\ref{thm} proves that nothing less than the full power of soundness of the system guarantees the truth of all such sentences.

\appendix\label{append}
\section*{Appendices.}
We present some refinements and generalizations of some technical results presented in the paper. For the sake of brevity, most of the proofs are omitted.

\section{More self-referential sentences.}\label{app:a}
G\"odel's proof made the self-referential sentence `I am unprovable' famous; his proof of the first incompleteness theorem shows that, for certain theories, every such sentence is in fact unprovable. Next, in the early 1950s, Leon Henkin asked: What if a sentence says about itself that it is provable? L\"ob's answer \cite{Lob1955}  is that such sentences are in fact provable. And one moral of our paper is that, in the absence of further information, from the very fact of the (un)provability of the I-am-(un)provable sentence, its truth does not follow.

Call a predicate $\Delta$ {\em self-fulfilling} with respect to a given theory $T$ iff the following is the case:
every sentence $\psi$ which, with respect to $T$, says about itself that it is a $\Delta$, is indeed a $\Delta$.     Thus (see our two examples in Section~\ref{sec:details} above) {\em is provable} and {\em is unprovable} are both self-fulfilling.

Here is a more interesting self-fulfilling predicate: {\em is decidable}. Suppose that, in the eye of $T$, a sentence $\psi$ says about itself that it is axiomatically decidable, i.e.,   $T\vdash\psi\!\leftrightarrow\!{\rm Pr}_T(\#\psi)\!\vee\!{\rm  Pr}_T(\#[\neg\psi])$. Then we   have $T\vdash\neg\psi\!\leftrightarrow\!\neg{\rm  Pr}_T(\#\psi)\!\wedge\!\neg{\rm  Pr}_T(\#[\neg\psi])$, so that $T+\neg\psi$ proves the consistency of  $T+\neg\psi$. Hence, by G\"odel's second incompleteness theorem, $T+\neg\psi$ is {\em in}consistent and we have $T\vdash\psi$. By arithmetization, ${\rm Pr}_T(\#\psi)$ is true in $\mathbb{N}$, and so is ${\rm Pr}_T(\#\psi)\!\vee\!{\rm Pr}_T(\#[\neg\psi])$. Therefore, every sentence which, in the eye of $T$, says about itself that it is $T$-decidable is in fact $T$-decidable---more specifically, any such sentence is $T$-provable.

Let us say that $\Delta$ is {\em self-falsifying} with respect to a given theory $T$ iff every sentence which says in $T$ about itself that it is a $\Delta$, is indeed {\em not}-$\Delta$. Examples include {\em is refutable} and {\em is consistent with} predicates.

Our argument in \S\ref{sec:details} actually proves  the following more general proposition.

\begin{proposition}\label{thm-app}
\textit{Let $T$ be a consistent theory (containing Robinson's arithmetic $\textit{Q}$). Let $\Delta$ be a self-fulfilling predicate and let $\Theta$ be a self-falsifying predicate with respect to $T$. Then the following are equivalent:
\begin{itemize}\itemindent=1.25em
\item[\textit{(i)}]  	The soundness of $T$.
\item[\textit{(ii)}]  	The truth of all the sentences which assert inside $T$ that they are $\Delta$.
\item[\textit{(iii)}]  	The falsehood of all the sentences asserting inside $T$ that they are $\Theta$.\footnote{The experts are invited to compare this to Theorem 24.7 of \cite[p.~182]{Smith2013}.}
\end{itemize}
}\end{proposition}
We remark, without offering a proof, that there are predicates that are neither self-fulfilling nor self-falsifying---examples include the predicates {\em is a universal sentence} and {\em is an existential sentence}.

\section{False G\"odelian sentences.}\label{app:b}
In \ref{subsubsec:truesound} above,
we proved that a theory is sound if and only if all its G\"odelian sentences are true. Perhaps more interestingly, every false sentence can be a G\"odelian sentence of a sufficiently strong sound theory. This follows from the next proposition.
\begin{proposition}
\textit{A sentence is $T$-unprovable if and only if it is a G\"odelian sentence of a consistent extension  of $T$.}
\end{proposition}
\begin{corollary}
\textit{For every sound theory $S$ and every false sentence $\varphi$, there exists a consistent extension $T$ of $S$ such that $\varphi$ is a G\"odelian sentence of  $T$.}
\end{corollary}

\section{More on $\boldsymbol\omega$-consistency.}\label{app:c}
Which levels of soundness are guaranteed by $\omega$-consistency? Which G\"odelian sentences of $\omega$-consistent theories are true? Our answer generalizes \cite[Theorem 17]{Isaacson2011}.
\begin{proposition}\label{prop}{\em
Let $T$ be an $\omega$-consistent extension of $\textit{Q}$. Then every $T$-provable $\Pi_3$-sentence is true, and so is every G\"odelian $\Pi_3$-sentence of $T$.
}\end{proposition}
This is a boundary result, since, as we have already mentioned in \ref{subsec:omega}, $\omega$-consistent theories may have false provable $\Sigma_3$-sentences; see \cite[Proposition 19]{Isaacson2011}. By Lemma~\ref{lem} in  \ref{subsubsec:truesound} above,  the conjunction of that provable false $\Sigma_3$-sentence with an arbitrary G\"odelian $\Pi_1$-sentence results in a false G\"odelian $\Sigma_3$-sentence. Hence $\omega$-consistent theories may have false G\"odelian $\Sigma_3$-sentences, though all of their G\"odelian $\Pi_3$-sentences are true by Proposition~\ref{prop}.

We hope by now the reader shares our view concerning the fallacious (or enthymematic) character of the passage we quoted from \cite[p.~151]{Godel1931} whose logical form we recognized as the $(1)$-$(2)$-$(3)$ scheme---he or she may now produce what G\"odel should have written instead (say by adding a premise about the complexity of the relevant predicate and sentence, or using a kind of direct self-reference \`{a} la Kripke).

Let us conclude by noting that G\"odel's argument does in fact contain a {\em redundant} premise. Recall that we formalized G\"odel's argument thus:
\begin{itemize}\itemindent=1em
\item[(4)] $T\vdash G\!\leftrightarrow\!\neg{\rm Pr}_T(\#G)$
\item[(5)] $\mathbb{N}\vDash\neg{\rm Pr}_T(\#G)$
\item[]   {\em therefore}
\item[(6)] $\mathbb{N}\vDash G$,
\end{itemize}
Now what we said in this Appendix shows that (5) is actually redundant, for {\em is unprovable} is a self-fulfilling predicate for consistent theories.


\begin{thebibliography}{99}

\bibitem{Boolos1990}
{\rm Boolos, George Stephen} (1990);
{\rm On ``seeing'' the truth of the G\"odel sentence},
{\em Behavioral and Brain Science}  13:
655--656.
Reprinted in: R. Jeffrey (ed.), G. Boolos, {\em Logic, Logic, and Logic}, Harvard University Press (1999), pp. 389--391.
 \textsc{doi}:~\href{https://doi.org/10.1017/S0140525X00080687}{\textsf{\small   10.1017/S0140525X00080687}}



\bibitem{BoolosJeffrey1989}
{\rm Boolos, George Stephen \& Richard Carl Jeffrey} (1989);
{\em Computability and Logic},
3rd ed., Cambridge University Press.



\bibitem{Cook2006}
{\rm Cook, Roy T. } (2006);
{\rm There are non-circular paradoxes (but Yablo's isn't one of them!)},
{\em The Monist} 89:
118--149.
 \textsc{doi}:~\href{https://doi.org/10.5840/monist200689137}{\textsf{\small   10.5840/monist200689137}}


\bibitem{Dawson1984}
{\rm Dawson, John William, the Junior} (1984);
{\rm The reception of G\"odel's incompleteness theorems},
{\em Proceedings of the Biennial Meeting of the Philosophy of Science Association} 1984:
253--271.
Reprinted in: T. Drucker (ed.), {\em Perspectives on the History of Mathematical Logic} (1991) pp. 84--100.
  \textsc{doi}:~\href{https://doi.org/10.1007/978-0-8176-4769-8_7}{\textsf{\small   10.1007/978-0-8176-4769-8\_7}}



\bibitem{Dummett1963}
{\rm Dummett, Michael Anthony Eardley} (1963);
{\rm The philosophical significance of G\"odel's theorem},
{\em Ratio} 5:
140--155.
Reprinted in: Dummett, {\em Truth and other Enigmas}, Duckworth, 1978, pp. 186--201.






\bibitem{Godel1931}
{\rm G\"odel, Kurt} (1931);
{\rm On formally undecidable propositions of {\em Principia mathematica} and related systems I.}
Reprinted in: \cite[pp.~145--195]{Godel1986}.



\bibitem{Godel1931a}
{\rm G\"odel, Kurt} (1931a);
{\rm Discussion on providing a foundation for mathematics}.
Reprinted in: \cite[pp.~201--205]{Godel1986}.


\bibitem{Godel1934}
{\rm G\"odel, Kurt} (1934);
{\rm On undecidable propositions of formal mathematical systems}.
Reprinted in: \cite[pp.~346--371]{Godel1986}.


\bibitem{Godel1986}
{\rm G\"odel, Kurt} (1986);
{\em Kurt G\"odel Collected Works, Volume I: Publications 1929--1936},
edited by S. Feferman {\em et al.}, Oxford University Press.



\bibitem{HalbachVisser2014}
{\rm Halbach, Volker \& Albert Visser} (2014);
{\rm Self-reference in arithmetic I \& II},
{\em The Review of Symbolic Logic} 7:
671--691\&692--712.
  \textsc{doi}:~\href{https://doi.org/10.1017/S1755020314000288}{\textsf{\small   10.1017/S1755020314000288}}, \href{https://doi.org/10.1017/S175502031400029X}{\textsf{\small   10.1017/S175502031400029X}}




\bibitem{Helmer1937}
{\rm Helmer, Olaf} (1937);
{\rm Perelman versus G\"odel},
{\em Mind} 46:
58--60.
  \textsc{doi}:~\href{https://doi.org/10.1093/mind/XLVI.181.58}{\textsf{\small   10.1093/mind/XLVI.181.58}}


\bibitem{HorstenWelch2016}
{\rm Horsten, Leon \& Philip Welch}, (eds.) (2016);
{\em G\"odel's Disjunction: The Scope and Limit of Mathematical Knowledge}, Oxford University Press.


\bibitem{Isaacson2011}
{\rm Isaacson, Daniel} (2011);
{\rm Necessary and sufficient conditions for undecidability of the G\"odel sentence and its truth}, in: D. DeVidi et al. (eds.),  {\em Logic, Mathematics, Philosophy, Vintage Enthusiasms: Essays in Honour of John L. Bell}, Springer, pp. 135--152.
  \textsc{doi}:~\href{https://doi.org/10.1007/978-94-007-0214-1_7}{\textsf{\small   10.1007/978-94-007-0214-1\_7}}






\bibitem{Kripke1975}
{\rm Kripke, Saul Aaron} (1975);
{\rm Outline of a theory of truth},
{\em The Journal of Philosophy} 72:
690--716.
\textsc{doi}:~\href{https://doi.org/10.2307/2024634}{\textsf{\small   10.2307/2024634}}






\bibitem{Kripke2014}
{\rm Kripke, Saul Aaron} (2014);
{\rm The road to G\"odel}, in: J. Berg (ed.),
{\em Naming, Necessity, and More: Explorations in the Philosophical Work of Saul Kripke},  Palgrave MacMillan, pp. 223--241.



\bibitem{LajevardiSalehi2019}
{\rm  Lajevardi, Kaave \& Saeed Salehi} (2019);
{\rm On the arithmetical truth of self‐referential sentences},
{\em Theoria: A Swedish Journal of Philosophy}
85:
8--17.
\textsc{doi}:~\href{https://doi.org/10.1111/theo.12169}{\textsf{\small   10.1111/theo.12169}}





\bibitem{LajevardiSalehi2021}
{\rm  Lajevardi, Kaave \& Saeed Salehi} (2021);
{\rm There may be many arithmetical  G\"odel sentence{\em s}},
{\em Philosophia Mathematica}
29:
278--287.
\textsc{doi}:~\href{https://doi.org/10.1093/philmat/nkaa041}{\textsf{\small   10.1093/philmat/nkaa041}}




\bibitem{LajevardiSalehi2023}
{\rm  Lajevardi, Kaave \& Saeed Salehi} (2023);
{\rm Soundness does not come for free (if at all)},
{\em PhilPapers} (Online Research in Philosophy):SALSDN.
\textsc{doi}:~\href{https://doi.org/10.48550/arXiv.2310.13422}{\textsf{\small   10.48550/arXiv.2310.13422}}




\bibitem{Leitgeb2002}
{\rm   Leitgeb, Hannes} (2002);
{\rm What is a self-referential sentence? critical remarks on the alleged \\ (non-)circularity of Yablo's paradox},
{\em Logique et Analyse} 45:
3--14.
\textsc{jstor}:~\href{https://www.jstor.org/stable/44084705}{\textsf{\small   44084705}}





\bibitem{Lindstrom1997}
{\rm Lindstr\"om, Per} (1997);
{\em Aspects of Incompleteness},  Springer.



\bibitem{Lob1955}
{\rm L\"ob, Martin Hugo} (1955);
{\rm Solution of a problem of Leon Henkin},
{\em The Journal of Symbolic Logic}
20:
115--118.
\textsc{doi}:~\href{https://doi.org/10.2307/2266895}{\textsf{\small   10.2307/2266895}}






\bibitem{McGee1992}
{\rm McGee, Van} (1992);
{\rm Maximal consistent sets of instances of Tarski's scheme (T)},
{\em Journal of Philosophical Logic}
21:
235--241.
\textsc{doi}:~\href{https://doi.org/10.1007/BF00260929}{\textsf{\small   10.1007/BF00260929}}





\bibitem{Mendelson1979}
{\rm  Mendelson, Elliott} (1979);
{\em Introduction to Mathematical Logic},
2nd ed., D. van Nostrand Company.




\bibitem{Milne2007}
{\rm Milne, Peter} (2007);
{\rm On G\"odel sentences and what they say},
{\em Philosophia Mathematica} 15:
193--226.
\textsc{doi}:~\href{https://doi.org/10.1093/philmat/nkm015}{\textsf{\small   10.1093/philmat/nkm015}}




\bibitem{Raatikainen2005}
{\rm Raatikainen, Panu} (2005);
{\rm On the philosophical relevance of G\"odel's incompleteness theorems},
{\em Revue Internationale de Philosophie} 234:
513--534.
\textsc{doi}:~\href{https://doi.org/10.3917/rip.234.0513}{\textsf{\small   10.3917/rip.234.0513}}




\bibitem{Rosser1936}
{\rm Rosser, John Barkley, the Senior} (1936);
{\rm Extensions of some theorems of G\"odel and Church},
{\em The Journal of Symbolic Logic} 1:
87--91.
\textsc{doi}:~\href{https://doi.org/10.2307/2269028}{\textsf{\small   10.2307/2269028}}





\bibitem{Sereny2011}
{\rm Ser\'eny, Gy\"orgy} (2011);
{\rm How do we know that the G\"odel sentences of a consistent theory is true?},
{\em Philosophia Mathematica} 19:
47--73.
\textsc{doi}:~\href{https://doi.org/10.1093/philmat/nkq028}{\textsf{\small   10.1093/philmat/nkq028}}




\bibitem{Shapiro1998}
{\rm Shapiro, Stewart} (1998);
{\rm Induction and indefinite extensibility: the G\"odel sentence is true, but did someone change the subject?},
{\em Mind} 107:
597--624.
\textsc{doi}:~\href{https://doi.org/10.1093/mind/107.427.597}{\textsf{\small   10.1093/mind/107.427.597}}




\bibitem{Smith2013}
{\rm Smith, Peter} (2013);
{\em An Introduction to G\"odel's Theorems},
2nd ed., Cambridge University Press.



\bibitem{Smorynski1977}
{\rm Smory\'nski, Craig} (1977);
{\rm The incompleteness theorems},
in: J. Barwise (ed.), {\em Handbook of Mathematical Logic}, North-Holland, pp. 821--865.




\bibitem{Smorynski1985}
{\rm Smory\'nski, Craig} (1985);
{\em Self-Reference and Modal Logic},
Springer.


\bibitem{Tennant2002}
{\rm Tennant, Neil} (2002);
{\rm Deflationism and the G\"odel phenomena},
{\em Mind} 111:
551--582.
\textsc{doi}:~\href{https://doi.org/10.1093/mind/111.443.551}{\textsf{\small   10.1093/mind/111.443.551}}



\bibitem{Visser2004}
{\rm Visser, Albert} (2004);
{\rm Semantics and the liar paradox}, in: D.M. Gabbay  and F. Guenthner (eds.),
{\em Handbook of Philosophical Logic}, 2nd ed., Volume XI, pp. 149--240.
\textsc{doi}:~\href{https://doi.org/10.1007/978-94-017-0466-3_3}{\textsf{\small   10.1007/978-94-017-0466-3\_3}}

\end{thebibliography}
\end{document}